\documentclass[11pt]{article}
\pdfoutput=1
\usepackage{amssymb,amsthm,amsmath}
\usepackage{graphicx}
\usepackage[small]{caption}
\usepackage{subcaption}
\usepackage{epsfig}

\newcommand{\later}[1]{}
\newcommand{\old}[1]{}

\usepackage{amsfonts}
\usepackage{booktabs}

\usepackage[utf8]{inputenc}
\usepackage{fullpage}
\usepackage{framed}

\usepackage{enumerate}
\usepackage{url}
\usepackage{hyperref}
\usepackage{tikz}
\usepackage{contour}

\newtheorem{lemma}{Lemma}

\newtheorem{proposition}{Proposition}

\newtheorem{definition}{Definition}
\usepackage{thm-restate}
\usepackage{xcolor}

\newcommand{\etal}{{et~al.}}

\newcommand{\AnyOrder}{{\rm AnyOrder}}

\newcommand{\NN}{\mathbb{N}} 

\title{Ordered Yao graphs: maximum degree, \\
edge density, and clique numbers 
}
\date{}

\author{P\'eter \'Agoston\thanks{Alfréd Rényi Institute of Mathematics, Budapest, Hungary and Charles University, Prague, Czech Republic\@.
    Email: \texttt{agostonp@renyi.hu}} \and
  \!\!\!\!Adrian Dumitrescu\thanks{
   Algoresearch L.L.C., Milwaukee, WI, USA, and 
Research Institute of the University of Bucharest, Romania, and 
Alfr\'ed R\'enyi  Institute of Mathematics, Budapest, Hungary.
  E-mail: \texttt{ad.dumitrescu@algoresearch.org}} \and
  \!\!\!\!Arsenii Sagdeev\thanks{Karlsruhe Institute of Technology, Karlsruhe, Germany and Alfréd Rényi Institute of Mathematics, Budapest, Hungary\@.
    Email: \texttt{sagdeevarsenii@gmail.com}} \and
  \!\!\!\!Karamjeet Singh\thanks{Indraprastha Institute of Information Technology, Delhi, India. Email: \texttt{karamjeets@iiitd.ac.in}} \and
  \!\!\!\!Ji Zeng\thanks{University of California San Diego, La Jolla, CA, USA and Alfréd Rényi Institute of Mathematics, Budapest, Hungary. Email: \texttt{jzeng@ucsd.edu}}
}

\begin{document}

\maketitle

\begin{abstract}
For a positive integer $k$ and an ordered set of $n$ points in the plane, define its \textit{k-sector ordered Yao graphs} as follows. Divide the plane around each point into $k$ equal sectors and draw an edge from each point to its closest predecessor in each of the $k$ sectors. We analyze several natural parameters of these graphs. Our main results are as follows:

\begin{enumerate} [I]

\item Let $d_k(n)$ be the maximum integer so that for every $n$-element point set in the plane, there exists an order such that the corresponding $k$-sector ordered Yao graph has maximum degree at least $d_k(n)$. We show that $d_k(n)=n-1$ if $k=4$ or $k \ge 6$, and provide some estimates for the remaining values of $k$. Namely, we show that $d_1(n) = \Theta( \log {n} )$; $\frac{1}{2}(n-1) \le d_3(n) \le 5\left\lceil\frac{n}{6}\right\rceil-1$; $\frac{2}{3}(n-1) \le d_5(n) \le n-1$;

\item Let $e_k(n)$ be the minimum integer so that for every $n$-element point set in the plane, there exists an order such that the corresponding $k$-sector ordered Yao graph has at most $e_k(n)$ edges. Then $e_k(n)=\left\lceil\frac{k}{2}\right\rceil\cdot n-o(n)$.

\item Let $w_k$ be the minimum integer so that for every point set in the plane, there exists an order such that the corresponding $k$-sector ordered Yao graph has clique number at most $w_k$. Then $\lceil\frac{k}{2}\rceil \le w_k\le \lceil\frac{k}{2}\rceil+1$.
\end{enumerate}

All the orders mentioned above can be constructed effectively.
\end{abstract}

\section{Introduction} \label{sec:intro}

For a point set in the plane, define its \textit{Yao graphs} in the following way. Fix an integer $k \ge 1$, and divide the plane around each point into $k$ equal sectors such that one boundary ray is horizontal and directed to the right. Then draw an edge from each point to its closest neighbor in each of the $k$ sectors, see Figure~\ref{FYaok}, left. Let us call the resulting directed graph a \textit{$k$-sector Yao graph}. Observe that the outdegree of every vertex is at most $k$, and is strictly smaller if some of the corresponding sectors are empty, as in Figure~\ref{FYaok}, left.

The notion of a \textit{$k$-sector ordered Yao graph} is closely related. In this case, the vertices appear one by one, and each new vertex has precisely one outgoing edge towards its closest predecessor in each of the $k$ sectors, see Figure~\ref{FYaok}, right. To make this notion well-defined, we also need to specify how to break the ties when a point lies on a sector-bounding ray or if two points are at the same distance from a third one and within the same sector of it. To simplify the arguments, all point sets considered are \textit{in general position}, i.e., neither of the scenarios described above occurs. However, as we briefly discuss in Section~\ref{sec:conc}, all our results hold for degenerate sets as well.

\begin{figure}[hbtp]
    \centering
        \includegraphics{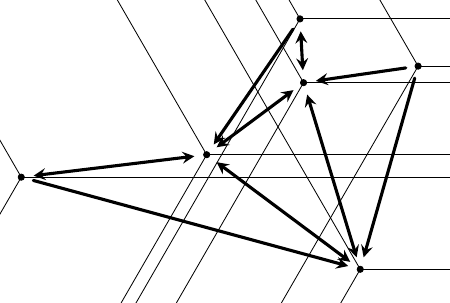}
        \includegraphics{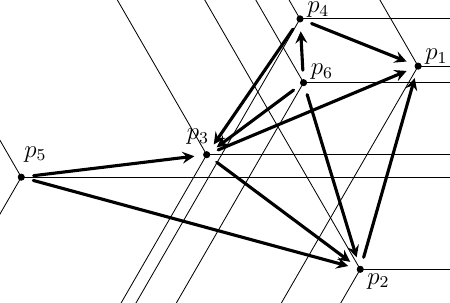}
	    \caption{3-sector unordered (left) and ordered (right) Yao graphs on the same set of six points.
 }
	\label{FYaok}
\end{figure}

Yao graphs were introduced by Yao~\cite{Yao82}, while their ordered variants are due to Bose, Gudmundsson, and Morin~\cite{BGM04} (defined for the slightly different variant called theta graphs). In modern Computational Geometry, these graphs are used to construct geometric spanners with desirable properties, such as logarithmic maximum degree and logarithmic diameter. In~\cite{BMN+04}, Bose et al. showed that the \textit{stretch factor} of $k$-sector Yao graphs is at most $1/(\cos(2\pi/k)-\sin(2\pi/k))$ for $k\ge9$, which has been improved several times, for example in \cite{BBD+15}. In \cite{BGM04}, the authors also study the stretch factor for the ordered variant. Sparse graphs with a small dilation have also been studied in~\cite{ABC+08,BGS05,BMN+04}, among others.

Note that Yao graphs in the special case $k=1$ are the well-known Nearest Neighbor Graphs with numerous applications in Computational Geometry such as computing geometric shortest paths, spanners, well-separated pairs, and approximate minimum spanning trees, see the survey~\cite{Sm00}, books~\cite{BCK+08,PS85}, or monograph~\cite{NS07}. The ordered variant of the Nearest Neighbor Graphs was introduced in~\cite{AEM92,Epp92} in the context of dynamic algorithms.

A systematic study of the basic combinatorial properties of Nearest Neighbor Graphs dates back at least to the classical paper~\cite{EPY97} by Eppstein, Paterson, and Yao in which, among others, they made the following simple observation: two edges with the same endpoint meet at an angle of at least $\pi/3$, and thus the maximum indegree is bounded from above by $6$ for planar point sets. However, it is not hard to see that for any $k>1$, the maximum indegree of a $k$-sector Yao graph can be arbitrarily large. Determining whether every point set $P$ admits an order such that the maximum indegree of the corresponding ordered Yao graph is bounded from above by some constant $c_k$, independent of the size of $P$, 
remains an open problem; see~\cite{BGM04}. To get a better understanding of how degrees behave in these graphs, here we attempt to \textit{maximize} the maximum indegree. 
The special case $k=1$ was addressed in a separate paper~\cite{ADSSZ24+}.
\begin{definition} \label{def:d_k}
For $k,n \in \NN$, let $d_k(n)$ be the maximum indegree one can always guarantee in an $n$-vertex $k$-sector ordered Yao graph by picking a suitable order. In other words, $d_k(n)$ is the maximum integer satisfying the following property. For every $n$-point set in the plane, there exists an order such that the corresponding $k$-sector ordered Yao graph has maximum indegree at least $d_k(n)$.
\end{definition}
	

\begin{restatable}{theorem}{maxdegYao}\label{thm:max_deg_Yao}
The following bounds hold:
	\begin{enumerate} \itemsep 1pt
		\item\label{item:d_1} $d_1(n) = \Theta( \log n )$;
		\item\label{item:d_3} $\frac{1}{2}(n-1) \le d_3(n) \le 5\left\lceil\frac{n}{6}\right\rceil-1$;
		\item\label{item:d_5} $\frac{2}{3}(n-1) \le d_5(n) \le n-1$;
		\item\label{item:gen_k} $d_k(n) = n-1$ otherwise, that is, if $k=2,4$ or $k\ge 6$.
	\end{enumerate}
\end{restatable}

Next we study the number of edges of a $k$-sector ordered Yao graph. This is trivial for $k=1$. Indeed, for every $n$-point set $P$ and every order of it, all the vertices of the corresponding ordered Yao graph, but the first one, have precisely one outgoing edge, and thus the graph always contains $n-1$ edges. Assume therefore that $k\ge 2$. Our next result determines the maximum number of edges of a $k$-sector ordered Yao graph that we \textit{can always guarantee} by picking a suitable order. We also obtain a `complementary' result regarding the maximum number of edges that we \textit{sometimes cannot avoid}, regardless of the order we take.

\begin{definition} \label{def:E_k}
For $k,n \in \NN$, let $E_k(n)$ be the maximum number of edges one can always guarantee in an $n$-vertex $k$-sector ordered Yao graph by picking a suitable order. In other words, $E_k(n)$ is the maximum integer satisfying the following property. For every $n$-point set in the plane, there exists an order such that the corresponding $k$-sector ordered Yao graph contains at least $E_k(n)$ edges.
\end{definition}


\begin{restatable}{theorem}{minmaxEG}
\label{prop:minmaxEG}
For $k \neq 3$, we have $E_k(n)=2n-3$ for $n \geq 3$,
and $E_3(n)=2n-4$ for $n \geq 4$.
\end{restatable}

\begin{definition} \label{def:e_k}
For $k,n \in \NN$, let $e_k(n)$ be the maximum number of edges one sometimes cannot avoid in an $n$-vertex $k$-sector ordered Yao graph regardless of the picked order. In other words, $e_k(n)$ is the maximum integer satisfying the following property. There exists an $n$-point set in the plane such that for every order, the corresponding $k$-sector ordered Yao graph contains at least $e_k(n)$ edges. 
\end{definition}


\begin{restatable}{theorem}{maxminEG}
\label{thm:maxminEG}
	For a fixed $k\ge 2$, we have $e_k(n) = n\cdot\left\lceil\frac{k}{2}\right\rceil -o(n)$ as $n \to \infty$. Moreover, we have $n\cdot\left\lceil\frac{k}{2}\right\rceil-O\left(k^2\cdot\sqrt{n}\right) \le e_k(n) \le n\cdot\left\lceil\frac{k}{2}\right\rceil -\Omega\left(k\cdot\sqrt{n}\right)$ for a fixed $k\geq 3$.
\end{restatable}

Finally, we study the \textit{clique number} of the $k$-sector ordered Yao graph, defined as the size of its largest subset of pairwise adjacent vertices, where we omit the orientation of the edges. As before, there is nothing to study if $k=1$. Indeed, all the vertices but the first one have precisely one outgoing edge, and thus the Yao graph is triangle-free, and in fact, it is not hard to show that the Yao graph is always \textit{acyclic}. So we can assume that $k \ge 2$. The following two `complementary' results determine the maximum size of a clique that we can always achieve by taking a suitable order of the point set, and estimate the maximum size of a clique that we sometimes cannot avoid, regardless of the order we take.

\begin{definition} \label{def:W_k}
For $k,n \in \NN$, let $W_k(n)$ be the maximum clique number one can always guarantee in an $n$-vertex $k$-sector ordered Yao graph by picking a suitable order. In other words, $W_k(n)$ is the maximum integer satisfying the following property. For every $n$-point set in the plane, there exists an order such that the corresponding $k$-sector ordered Yao graph contains a clique of size $W_k(n)$.
\end{definition}


\begin{restatable}{theorem}{cliqueminmax}
\label{th:clique_minmax}
For $k \geq 2$ and $n \geq 3$ we have $W_k(n)=3$ with the only exception being $W_3(3)=2$.
\end{restatable}

\begin{definition} \label{def:w_k}
For $k,n \in \NN$, let $w_k(n)$ be the maximum clique number one sometimes cannot avoid in an $n$-vertex $k$-sector ordered Yao graph regardless of the picked order, and $w_k = \max_n w_k(n)$. In other words, $w_k$ is the maximum integer satisfying the following property. There exists a point set in the plane such that for every order, the corresponding $k$-sector ordered Yao graph contains a clique of size $w_k$.
\end{definition}


\begin{restatable}{theorem}{maxminomegaG}
\label{maxminomegaG}
For $k \geq 2$ we have $\lceil\frac{k}{2}\rceil \leq w_k \leq \lceil\frac{k}{2}\rceil+1$.
\end{restatable}

It is not hard to see that, as a function of $n$, $w_k(n)$ is monotonically non-decreasing and bounded from above by $k+1$, the maximum outdegree increased by $1$.  Hence, $w_k(n)=w_k$ for all sufficiently large $n$, and in Section~\ref{sec:conc}, we make some observations regarding smaller values of $n$.

\paragraph{Motivation and related work.}
In this paper, we define closest neighbors based on the Euclidean distance. However, there are alternative ways. For instance, one may want to minimize the distance between a point and the orthogonal projection of its neighbor on the bounding ray of the hosting sector (or, in another variant, on the bisector of the hosting sector). The resulting graphs, usually referred to in the literature as \textit{$\theta$-graphs}, were introduced by Clarkson~\cite{Cla87} and independently by Keil~\cite{Keil88}, while their ordered variants are due to Bose~\etal~\cite{BGM04}. The spanning ratio of these graphs is at most $1/(1- 2 \sin(\pi/k))$ for $k \geq 7$ and any order of the point set~\cite{Ren14}, see also~\cite{ABC+08,Dam18,DR12}.

Our interest in studying the extremal values of the basic graph parameters of ordered Yao graphs is mostly theoretical. However, these extremal questions can be also viewed as measuring the robustness of incremental cone-based geometric networks under adversarial insertion orders. This connects naturally to the literature on online Euclidean~\cite{bhore2024online2} and metric~\cite{bhore2024online} spanners, where points arrive one by one and the goal is to preserve sparsity and good routing quality despite the arrival order.


\subsection{Notation}

We assume that one of the sector-bounding rays of the $k$-sector Yao graph is a horizontal ray directed to the right and call it $\ell_0$, and the remaining rays are $\ell_1,\ell_2,...,\ell_{k-1}$ in a counterclockwise order from $\ell_0$. Denote the sectors between these rays by $s_0,s_1,...,s_{k-1}$ in  counterclockwise order. For instance, $s_0$ is the sector bounded by $\ell_0$ and $\ell_1$, and so forth. Let us call the antipodal rays $-\ell_0, \dots, -\ell_{k-1}$ the \textit{dual rays}. Drawn from a common origin, they separate the plane into $k$ \emph{dual sectors} labeled by $-s_0,...,-s_{k-1}$. When centered around a point $p$, these notions will be denoted by $\ell_i(p)$, $s_i(p)$, $-\ell_i(p)$ and $-s_i(p)$, respectively. It is easy to check that $p\in s_i(q)$ if and only if $q\in-s_i(p)$, see Figure~\ref{F1}.

The notation $\AnyOrder(Q)$ stands for an arbitrary order of the point set $Q$.

\begin{figure}[hbtp]
\centering
\includegraphics{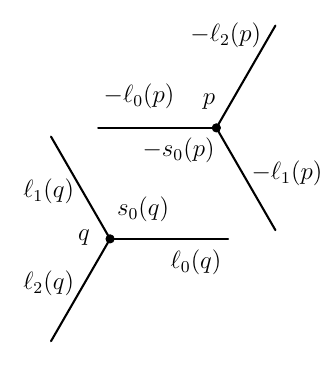}
\caption{$p\in s_0\left(q\right)$ if and only if $q \in-s_0\left(p\right)$.}
\label{F1}
\end{figure}

\section{Maximum degree} \label{sec:max_deg}

In this section we prove Theorem~\ref{thm:max_deg_Yao}.
Let us briefly outline the strategy of the proof. First, note that the equality $d_1(n) = \Theta(\log n)$ is immediate from Theorem 1 and Theorem 2 in \cite{ADSSZ24+}. Indeed, the authors in \cite{ADSSZ24+} considered ordered Nearest Neighbor graphs, which coincide with $k$-sector ordered Yao graph for $k=1$. The authors proved that any set of $n$ points on a line admits an ordering such that the corresponding ordered Nearest Neighbor graph has maximum indegree at least $\lceil\log{n}\rceil$. Here and in the rest of this paper, $\log{n}$ stands for the binary logarithm. The bound was shown to be tight by constructing an explicit set of $n$ points for which the maximum indegree cannot exceed $\lceil\log{n}\rceil$ irrespective of the ordering of the points. 
The result was extended to $\mathbb{R}^d$ and it was shown that any $n$-point set in $\mathbb{R}^d$ admits an ordering ensuring the maximum indegree of the corresponding ordered  Nearest Neighbor graph at least $\log{n}/(4d)$.
So we subsequently assume that $k>1$. In Subsection~\ref{subsec:orth}, we present a simple order yielding that $d_k(n) = n-1$ for all even $k$, and that $d_k(n) \ge \frac{1}{2}(n-1)$ for all odd $k\ge 3$. In Subsection~\ref{subsec:rad}, we give another simple order that yields $d_k(n) = n-1$ for all $k \ge 6$. 
A combination of these two orders implies that $d_5(n) \ge \frac{2}{3}(n-1)$. 
Finally, in Subsection~\ref{subsec:d3}, we prove that $d_3(n) \le 5\left\lceil\frac{n}{6}\right\rceil-1$ via an explicit construction. Together these results complete the proof of Theorem~\ref{thm:max_deg_Yao}.

\subsection{Orthogonal enumeration} \label{subsec:orth}
    
\begin{lemma} \label{lemma:orthgonal}
 Let $Q$ be a point set contained in $t=\lfloor\frac{k}{2}\rfloor$ cyclically consecutive dual sectors of a point $p$. Then there exists an ordering such that the corresponding $k$-sector Yao graph contains all the directed edges $q \to p$ for $q\in Q$.
\end{lemma}
\begin{proof}
  Due to the symmetry of this statement, it suffices to prove it only for the points from dual sectors between $-\ell_0$ and $-\ell_t$. Let $q_1,\dots,q_m$ be these points labeled in increasing order of their $y$-coordinates. 
We claim that in the $k$-sector Yao graph corresponding to the ordering
\[p,q_1,\dots,q_m, \AnyOrder(P\setminus \{p,q_1,\dots,q_m\}),\]
each $q_i$ is adjacent to $p$. Indeed, all the previous $q_j, j < i,$ lie below $q_i$ by construction, see Figure~\ref{F_orth}.
At the same time, $p$ belongs to one of the first $t$ sectors of $q_i$, and thus $p$ lies above $q_i$ since $t \le \frac{k}{2}$. So $p$ is the unique point in one of the first $t$ sectors of $q_i$, and thus  $q_i \to p$ is an edge. 
\end{proof}


\begin{proposition} \label{prop:even}
	For all $n \in \NN$, we have $d_k(n)=n-1$ if $k$ is even, and $d_k(n)\ge \frac{1}{2}(n-1)$ if $k>1$ is odd.
\end{proposition}
\begin{proof}

Let $p$ be the highest point of $P$. Note that $t=\lfloor\frac{k}{2}\rfloor$ of its dual sectors that lie in
the upper half-plane are empty by construction, see Figure~\ref{F_orth}.
	
If $k$ is even, then the remaining $n-1$ points are distributed between $t$ dual sectors in the lower half-plane,
and Lemma~\ref{lemma:orthgonal} implies that all of them can be adjacent to $p$ under a suitable ordering.
Therefore, $d_k(n)=n-1$, as desired.
	
If $k$ is odd, the remaining $n-1$ points are distributed between $t+1$ dual sectors intersecting the lower
half-plane. By the pigeonhole principle,  either the first $t$ or the last $t$ contain at least half of the points. Moreover, all of them can be adjacent to $p$ under a suitable ordering according to Lemma~\ref{lemma:orthgonal}. Therefore, $d_k(n) \ge \frac{1}{2}(n-1)$, as desired.
\end{proof}

\begin{figure}[htbp]
\centering
\includegraphics{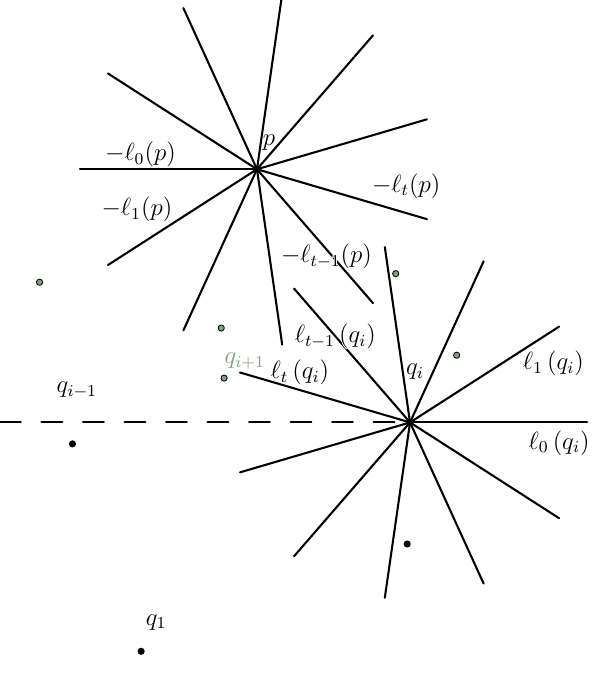}
\caption{An illustration to the proofs of Lemma~\ref{lemma:orthgonal} and Proposition~\ref{prop:even}.}
\label{F_orth}
\end{figure}

\subsection{Radial enumeration} \label{subsec:rad}
	
\begin{proposition} \label{prop:k>=6}
For all $n \in \NN$, $k \ge 6$, we have $d_k(n)=n-1$.
\end{proposition}
\begin{proof}
Let $P$ be any point set and $p$ be an arbitrary point of $P$. Label the remaining points $q_1,\dots,q_{n-1}$ such that their Euclidean distances to $p$ are nonincreasing. We claim that the indegree of $p$ in the $k$-sector Yao graph under the ordering $p,q_1,\dots,q_{n-1}$ equals $n-1$. Assume the contrary, namely that for some $i$, the point $q_i$ is not adjacent to $p$.
It follows that $q_i$ is adjacent to $q_j$ for some $j<i$, where $q_j$ and $p$ belong to the same sector around $q_i$. In particular, we have  $\measuredangle pq_iq_j < 2\pi/k \le \pi/3$, and thus $\angle pq_iq_j$ is not a largest angle of the triangle $q_ipq_j$. Hence, $pq_j$ is not a longest side of this triangle by the law of sines. Since we labeled the points in such an order that $|q_ip|\le|pq_j|$, we conclude that $q_iq_j$ is the unique longest side of the triangle $q_ipq_j$. However, in this case $q_i$ should be adjacent to $p$ instead of $q_j$ in the corresponding sector, a contradiction.
\end{proof}	
	
\begin{proposition} \label{prop:k=5}
	For all $n \in \NN$, we have $d_5(n)\ge \frac{2}{3}(n-1)$.
\end{proposition}
\begin{proof}
Let $P$ be any point set and $p$ be the highest point of $P$, so for $k=5$, $-s_3(p)$ and $-s_4(p)$ are empty. Let $P_0, P_1, P_2$ be sets of the remaining points in $-s_0(p)$, $-s_1(p)$ and $-s_2(p)$, respectively, and $a_0,a_1, a_2$ be their cardinalities, see Figure~\ref{F3}.

  
		\begin{figure}[!htb]
			\centering
			\includegraphics{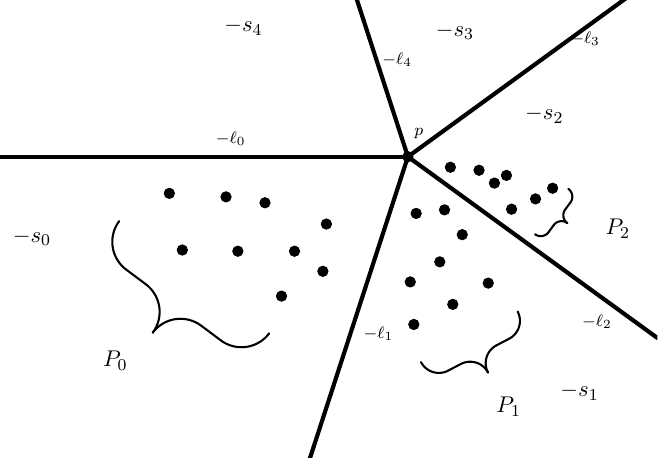}
			\caption{Five dual sectors of $p$; the last two are empty.}
			\label{F3}
		\end{figure}

	Lemma~\ref{lemma:orthgonal} implies that there exists an ordering such that each vertex of $P_0\cup P_1$ is adjacent to $p$. Besides, there exists an ordering such that each vertex of $P_1\cup P_2$ is adjacent to $p$. A similar statement for the union $P_0\cup P_2$ will complete the proof. Indeed, since $a_0+a_1+a_2=n-1$, at least one of the sums $a_0+a_1, a_1+a_2, a_2+a_0$ is at least $\frac{2}{3}(n-1)$ by the pigeonhole principle.
		
	Let $q_1,\dots,q_{m}$ be the points of $P_0\cup P_2$ labeled such that their Euclidean distances to $p$ are nonincreasing, where $m=a_2+a_0$. We claim that each of them is adjacent to $p$ in the $5$-sector Yao graph under the ordering\begin{equation*}
	    p,q_1,\dots,q_{m}, \AnyOrder(P_1).
	\end{equation*}
	
	As in the proof of Proposition~\ref{prop:k>=6}, assume the contrary, namely that for some $j<i$, $q_i$ is adjacent to $q_j$ instead of $p$. First, suppose that $q_i$ and $q_j$ belong to different dual sectors of $p$. This yields that $\measuredangle q_ipq_j > \frac{2\pi}{5} > \frac{\pi}{3}$. As earlier, we conclude that $\angle q_ipq_j$ is not a smallest angle of the triangle $q_ipq_j$ and $q_iq_j$ is not a shortest side. Hence, $|q_ip|<|q_iq_j|$ and thus $q_i$ should be adjacent to $p$ instead of $q_j$, a contradiction.
	
	Second, suppose that $q_i$ and $q_j$ belong to the same dual sectors of $p$, say, $q_i,q_j \in P_2$. Then $q_j\in-s_2(q_i)\cap s_2(p)$. This intersection is a parallelogram with angles $\frac{2\pi}{5} < \frac{\pi}{2}$ at the vertices $q_i$ and $p$. Therefore, $q_ip$ is the diameter of this parallelogram, and thus $|q_ip|>|pq_j|$, a contradiction again.
\end{proof}

\subsection{Upper bound on $d_3(n)$} \label{subsec:d3}

\begin{proposition}\label{prop:d3}
We have $d_3(n)\le5\left\lceil\frac{n}{6}\right\rceil-1$.
\end{proposition}

\begin{proof}
Since $d_3(n)$ is clearly nondecreasing as a function of $n$, assume without loss of generality that $n=6m$ for some $m\in \NN$ and construct an $n$-element point set as follows. 
Pick a very small angle, say, $\alpha =\frac{\pi}{10m}$. For $1\le i \le m$, 
 let $a_i, b_i, c_i, d_i, e_i$, and $f_i$ be points on the unit circle whose angles with
        the $x$-axis equal to $i\alpha$, $-i\alpha$,
        $\frac{2\pi}{3}+i\alpha$, $\frac{2\pi}{3}-i\alpha$,
        $\frac{4\pi}{3}+i\alpha$, and $\frac{4\pi}{3}-i\alpha$,
        respectively, see Figure~\ref{F4} (recall that $k=3$). 
    A small perturbation brings it in general position.
	
	\begin{figure}[htbp]
		\centering
        \resizebox{0.47\textwidth}{!}{\includegraphics{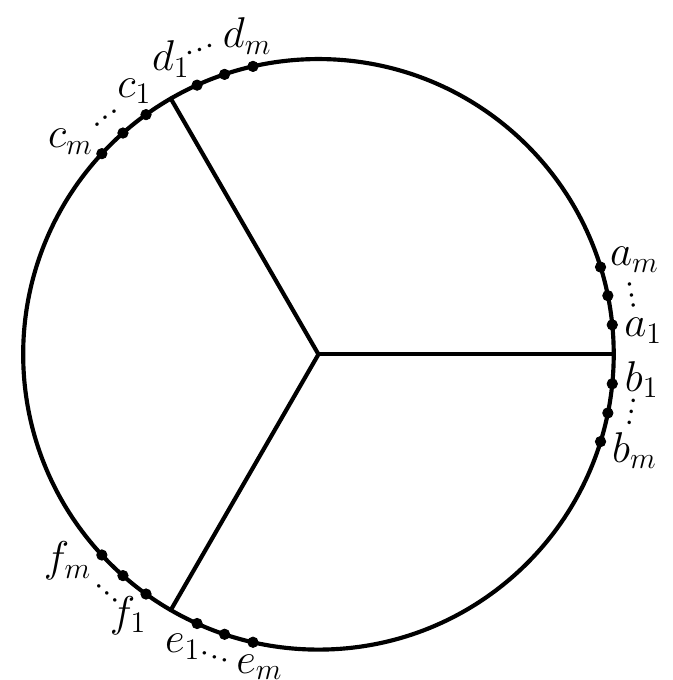}}
		\caption{Both $c_m$ and $a_i$ belong to the first sector of $f_1$, both $f_1$ and $a_i$
                  belong to the third sector of $c_m$, $c_mf_1$ is the shortest side of the triangle $a_ic_mf_1$.}
		\label{F4}
	\end{figure}
	
	To show that $d_3(n) \le5\left\lceil\frac{n}{6}\right\rceil-1$, it suffices  to show that for every ordering, the indegree of each vertex in the $3$-sector Yao graph does not exceed
        $(n-1)-m$. Due to the symmetry of this configuration, it suffices 
        to show this only for $a_i$, $1 \le i \le m$. Pick
        $1 \le j \le m$ and consider a triangle
        $a_ic_{m+1-j}f_j$. Note that both $c_{m+1-j}$ and $a_i$ belong
        to the first sector of $f_j$, and that both $f_j$ and
        $a_i$ belong to the third sector of $c_{m+1-j}$, see
        Figure~\ref{F4}. Moreover, it is not hard to see that
        $c_{m+1-j}f_j$ is the shortest side of the triangle
        $a_ic_{m+1-j}f_j$. Therefore, among $c_{m+1-j}$ and $f_j$, the
        vertex that appears later cannot be adjacent to $a_i$. Since
        this holds for all $1\le j \le m$, we conclude that the
        indegree of $a_i$ does not exceed $(n-1)-m$ under each
        ordering, as desired.
\end{proof}

\vspace{3mm}

\begin{proof}[Proof of Theorem \ref{thm:max_deg_Yao}]
Recall, as mentioned at the beginning of Section \ref{sec:max_deg} that the equality $d_1(n) = \Theta( \log n )$ follows from Theorem 1 and Theorem 2 in \cite{ADSSZ24+}.
The bounds in Case~\ref{item:d_3} of the theorem follow from Proposition \ref{prop:even} and Proposition \ref{prop:d3}. By Proposition \ref{prop:k=5}, $\frac{2}{3}(n-1)\le d_5(n)$. Also, $d_5(n)\le n-1$ holds trivially. Finally, Case~\ref{item:gen_k} of the theorem follows from Proposition \ref{prop:even} and Proposition \ref{prop:k>=6}.  
\end{proof}

\section{Clique numbers}

In this section, we will examine the clique number of ordered Yao graphs. We use nothing about the distances among the points, thus, all of our arguments work for other variants, like ordered theta graphs.

\subsection{Maximizing the largest clique}

For the proof of Theorem \ref{th:clique_minmax}, we first give a construction of a set $P$ of $n$ points ensuring that, regardless of the ordering of the points in $P$, the clique number of the $k$-sector ordered Yao graph is at most $3$. We also give a specific 3-point set such that the corresponding $3$-sector ordered Yao graph is triangle-free regardless of the ordering. Then we show that any arbitrary set of $n$ points admits an ordering for which the $k$-sector ordered Yao graph contains a triangle, except when $k=n=3$.

\begin{proof}[Proof of Theorem~\ref{th:clique_minmax}]

To prove the upper bound, consider a set $P$ of $n$ points on a generic line (not parallel to any of the $\ell_i$). It is easy to see that each point contains all the others in (at most) two sectors, and thus its outdegree is at most two regardless of the ordering. Hence, the clique number of the $k$-sector ordered Yao graph is at most $3$.

For $k=n=3$, take a triangle centered at the origin and whose vertices are on $\ell_0$, $\ell_1$ and 
$\ell_2$, see Figure~\ref{fig:k=n=3}. It is easy to see that for every order, the corresponding $3$-sector ordered Yao graph contains only two edges, and thus it is triangle-free.

\begin{figure}[hbtp]
    \centering
    \includegraphics{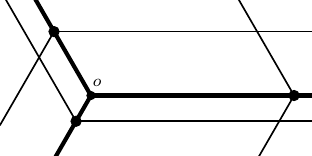}
	\caption{A set with a triangle-free $3$-sector ordered Yao graph.}
	\label{fig:k=n=3}
\end{figure}

Now we prove the lower bounds. Our goal is to find three points $p_1, p_2, p_3$ in an arbitrary point set $P$ such that one of them, say $p_3$, contains the other two in different sectors. For such a triple, the $k$-sector ordered Yao graph corresponding to the order $$p_1, p_2, p_3, \AnyOrder(P\setminus\{p_1,p_2,p_3\})$$ contains a triangle $p_1p_2p_3$. 

First, consider the case when $k$ is even. Among $3$ arbitrary points, there is one, which is neither topmost, nor bottommost and thus does not contain the other two in the same sector, and we are done.

Similarly, if $k\ge6$, then an arbitrary triangle from $P$ has one angle at least $\frac{\pi}{3} \ge \frac{2\pi}{k}$, and thus its vertex contains the other two in different sectors.

If $k=5$, we take three arbitrary points from $P$ and translate the angles of their triangle $\Delta$ to a common origin, which divides the plane into $3$ `cones' and $3$ `complementary cones', see Figure~\ref{fig:k=5}. 

\begin{figure}[hbtp]
	\centering
	\begin{minipage}{.49\textwidth}
	\includegraphics{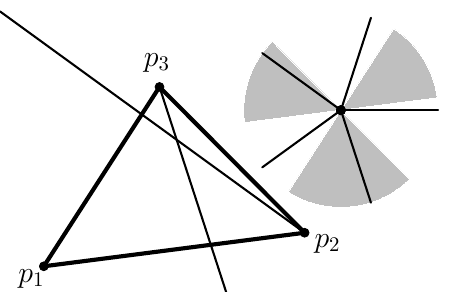}
	\end{minipage}
	\begin{minipage}{.49\textwidth}
	\includegraphics{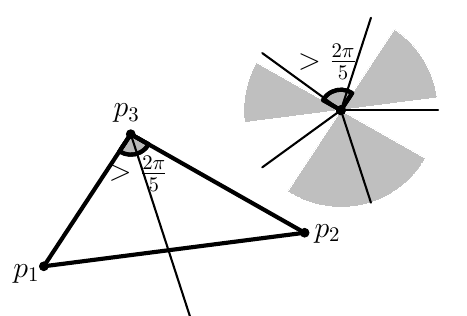}
	\end{minipage}
	\caption{Illustration for the case $k=5$ of Theorem \ref{th:clique_minmax}. On the left, we can see an occurrence of the first subcase when there are cones (colored in gray) which contain at least one of the $\ell_i$. On the right, we can see an occurrence of the second subcase when there is a complementary cone containing two of the $\ell_i$, which also results in the cone on the opposing side containing an $\ell_i$, or in other words, one of the $p_i$ seeing the other two of the $p_i$ in different sectors.}
	\label{fig:k=5}
\end{figure}

By the pigeonhole principle, either one of the cones contains at least one of the $\ell_i$'s, or at least one of the dual cones contains two. In the first case, we are done, because one of the three points has the other two in different sectors. In the second case, we are also done, since one of the angles of $\Delta$ is larger than $\frac{2\pi}{5}$, and thus it must contain an $\ell_i$.

Finally, suppose that $k=3$, $n\ge4$, and each point contains the others in one sector. By the pigeonhole principle, some two of the points contain all the others in the same sector, say, in $s_0$. However, for all $p_1,p_2 \in P$, if $p_2\in s_0\left(p_1\right)$, then $p_2$ is higher than $p_1$, and thus $p_1\notin s_0\left(p_2\right)$. This contradiction completes the proof.
\end{proof}

\subsection{Minimizing the largest clique}

\label{maxminomegaGupper}
\begin{proof}[Proof of Theorem~\ref{maxminomegaG}]
The upper bound is trivial: since the last $\left\lfloor\frac{k}{2}\right\rfloor$ of the sectors corresponding to each vertex belong to the lower half-plane, they do not contain the preceding points in the top-to-bottom ordering. Hence, the outdegree of each vertex is at most $\left\lceil\frac{k}{2}\right\rceil$ in the $k$-sector ordered Yao graph, and so the size of any clique is at most $\left\lceil\frac{k}{2}\right\rceil+1$.

For the lower bound, we construct a set $P$ of $\left\lceil\frac{k}{2}\right\rceil$ points in the plane such that every point contains all the others in pairwise distinct sectors. It is clear that no matter how we order these points, the corresponding $k$-sector ordered Yao graph would be a clique.

Our construction of such a set depends on the residue of $k$ modulo 4. For $k=2m$, we take the vertex set a regular $m$-gon centered at the origin, with one vertex initially on the positive horizontal axis, and then slightly rotate it about the origin, see Figure~\ref{FYaok2}~(right). In case $k=2m+1$ and $m$ is odd, we take the $0^{\mbox{th}}$, $4^{\mbox{th}}$, $8^{\mbox{th}}$,...,$(4m-4)^{\mbox{th}}$, and $(4m-1)^{\mbox{th}}$ vertices of a regular $(4m+2)$-gon labeled in a counterclockwise order such that the $0^{\mbox{th}}$ vertex is on the positive horizontal axis, and then slightly rotate this set, see Figure~\ref{FYaok2}~(middle). Finally, in case $k=2m+1$ and $m$ is even, we take the $1^{\mbox{st}}$, $5^{\mbox{th}}$, $9^{\mbox{th}}$,...,$(4m-3)^{\mbox{th}}$, and $4m^{\mbox{th}}$ vertices of a regular $(4m+2)$-gon labeled in a counterclockwise order such that the $0^{\mbox{th}}$ vertex is on the positive horizontal axis, and then slightly rotate this set, see Figure~\ref{FYaok2}~(left).

To complete the proof, it remains only to check that all three constructions described above satisfy the desired property, namely that every point contains all the others in pairwise distinct sectors. This task is straightforward yet tedious. To formally verify this, we utilize the standard bijection between the plane and the set of complex numbers $\mathbb{C}$.

\begin{figure}[hbtp]
    \centering
	\begin{minipage}{.32\textwidth}
        \resizebox{0.8\textwidth}{!}{\includegraphics{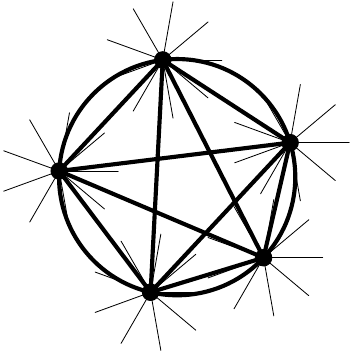}}
        \end{minipage}
	\begin{minipage}{.32\textwidth}
        \resizebox{0.8\textwidth}{!}{\includegraphics{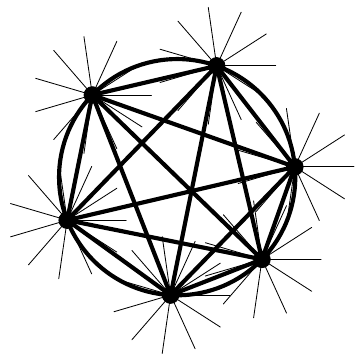}}
        \end{minipage}
        	\begin{minipage}{.32\textwidth}
        \resizebox{0.8\textwidth}{!}{\includegraphics{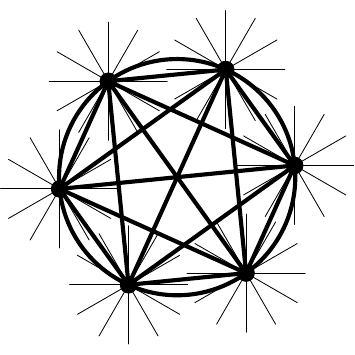}}
        \end{minipage}
	\caption{Examples for $k=9,11,12$.
 }
	\label{FYaok2}
\end{figure}

Let $z$ be the $2k$-th root of $1$ with the smallest positive argument, that is, $z=\cos\frac{\pi}{k} + i \cdot \sin\frac{\pi}{k}$. The directions of the rays $\ell_0,\ell_1,...,\ell_{k-1}$ are $z^0,z^2,...,z^{2k-2}$, respectively. In other words, even powers of $z$ correspond to sector boundaries. Observe that if $j_1-k<j_2<j_1+k$, i.e., if the angle between the directions $z^{j_2}$ and $z^{j_1}$ is strictly between $-\pi$ and $\pi$, then the direction of their sum $z^{j_2}+z^{j_1}$ is the bisector $z^{(j_1+j_2)/2}$. Hence, if $j_1<j_2<j_1+2k$, i.e., if the angle between the directions $z^{j_2}$ and $z^{j_1}$ is strictly between $0$ and $2\pi$, then the direction of the vector connecting $z^{j_1}$ to $z^{j_2}$, i.e. of the vector $z^{j_2}-z^{j_1} = z^{j_2}+z^{j_1+k}$ is $z^{(j_1+j_2+k)/2}$. Observe that the condition $j_1<j_2<j_1+2k$ here is crucial: if we switch the roles of $j_1$ and $j_2$ or consider a different representation $z^{j_2+2k}$ of the same point $z^{j_2}$, then the formula is only correct up to sign. In the rest of this proof, we will apply this formula several times, its simplicity somewhat justifies the usage of the complex plane.

In case $k=2m$, put $p_j = z^{4j}$ for $0 \le j < m$. Note that if $j_1<j_2$, then the direction of the vector $p_{j_2}-p_{j_1}$ is $z^{(4j_1+4j_2+2m)/2} = z^{2j_1+2j_2+m}$, thus the direction of the opposite vector $p_{j_1}-p_{j_2}$ is $z^{2j_1+2(j_2+m)+m}$. Now one can check that for each fixed $j_1$, the powers of the directions of the vectors $p_{j_2}-p_{j_1}$, $j_2\neq j_1$, form the set $I_{j_1} = \{2j_1+2j_2+m: j_1<j_2<j_1+m\}$.

If $m$ is even, then the elements of each $I_{j_1}$ are consecutive even integers. Hence, for each $j_1$, the directions of the vectors $p_{j_2}-p_{j_1}$, $j_2 \neq j_1$, coincide with $m-1$ cyclically consecutive sector boundaries. Rotating this configuration by a sufficiently small angle about the origin moves these directions into the interiors of $m-1$ cyclically consecutive sectors. That is $P=\{p_j\cdot z^\varphi: 0 \le j < m\}$ is the desired configuration for a sufficiently small $\varphi$, see Figure~\ref{FYaok2}~(right). 

If $m$ is odd, then the elements of each $I_{j_1}$ are consecutive odd integers. Hence, each $p_{j_1}$ contains all the other points in the interiors of $m-1$ cyclically consecutive sectors, exactly on their bisectors. In this case, a small rotation is not even needed, but it also cannot break the desired property. This completes the proof of Theorem~\ref{maxminomegaG} if $k$ is even.

In case $k=2m+1$ and $m$ is odd, we put $p_j = z^{4j}$ for $0 \le j < m$ and $p_m = z^{4m-1}$. We claim that after a sufficiently small rotation about the origin, every point contains all the others in pairwise distinct sectors. We begin by considering the special point $p_m$. For each $0\le j < m$, the direction of the vector $p_{m}-p_j$ is $z^{((4m-1)+4j+(2m+1))/2}=z^{2j+3m}$ and thus the direction of the opposite vector $p_{j}-p_m$ is $z^{(2j+3m)-(2m+1)}=z^{2j+m-1}$. The powers $\{2j+m-1: 0 \le j < m\}$ are consecutive even integers, and thus the directions of the vectors $p_{j}-p_m$, $0 \le j < m$, coincide with $m$ cyclically consecutive sector boundaries. Hence, rotating this configuration by a sufficiently small angle about the origin moves these directions into the interiors of $m$ cyclically consecutive sectors, as claimed, see Figure~\ref{FYaok2}~(middle).

Now we prove the claim for each fixed $0 \le j_1 < m$. For $j_1<j_2<m$, the direction of the vector $p_{j_2}-p_{j_1}$ is $z^{2j_1+2j_2+m+1/2}$. As we have already seen, the direction of the vector $p_m-p_{j_1}$ is $z^{2j_1+3m}$. Finally, for $0 \le j_2<j_1$, the direction of the vector $p_{j_1}-p_{j_2}$ is $z^{2j_1+2j_2+m+1/2}$ and thus the direction of the opposite vector $p_{j_2}-p_{j_1}$ is $z^{(2j_1+2j_2+m+1/2)+(2m+1)}=z^{2j_1+2j_2+3m+3/2}$. For $j_2 = j_1+1,j_1+2,...,m-1, m,0,1,..., j_1-1$, the powers of these directions are equal to $2j_1+2(j_1+1)+m+1/2,\ 2j_1+2(j_1+2)+m+1/2,..., 2j_1+2(m-1)+m+1/2,$  $2j_1+3m,$ $2j_1+2\cdot0+3m+3/2, \ 2j_1+2\cdot1+3m+3/2,..., 2j_1+2(j_1-1)+3m+3/2$. Observe that an open interval between any two consecutive numbers in this sequence is of length at most 2 and contains an even integer. Indeed, since all but two of the lengths are precisely equal to 2, it is sufficient to note that the open interval between $2j_1+2(m-1)+m+1/2$ and $2j_1+3m$ of length $3/2$ contains an even integer $2j_1+3m-1$, and that the open interval between $2j_1+3m$ and $2j_1+2\cdot0+3m+3/2$ of length $3/2$ contains an even integer $2j_1+3m+1$. Therefore, all the points $p_{j_2}$, $j_2\neq j_1$, belong to the interiors of pairwise distinct sectors of $p_{j_1}$, as claimed. Note that a rotation about the origin by a sufficiently small angle does not affect this property, see Figure~\ref{FYaok2}~(middle). This completes the proof of Theorem~\ref{maxminomegaG} if $k=2m+1$ and $m$ is odd.
        
The case $k=2m+1$ and $m$ is even, is very similar to the previous one with the only difference being the multiplication by $z$, which will increase the powers of all the directions by 1 to `compensate' the change of the parity of $m$. In other words, we put $p_j = z^{4j+1}$ for $0 \le j < m$ and $p_m = z^{4m}$. As before, we claim that after a sufficiently small rotation about the origin, every point contains all the others in pairwise distinct sectors. One can check as earlier that for each $0\le j < m$, the direction of the vector $p_{j}-p_m$ is $z^{2j+m}$, and their powers are consecutive even integers. Thus the directions of the vectors $p_{j}-p_m$, $0 \le j < m$, coincide with $m$ cyclically consecutive sector boundaries and a sufficiently small rotation moves these directions into the interiors of $m$ cyclically consecutive sectors, as claimed, see Figure~\ref{FYaok2}~(left). Similarly, for each $0 \le j_1 < m$, one can compute the directions of the vectors $p_{j_2}-p_{j_1}$ for $j_2 = j_1+1,j_1+2,...,m-1, m,0,1,..., j_1-1$ and observe that the sequence of their powers satisfies the following property: an open interval between any two consecutive numbers in this sequence is of length at most 2 and contains an even integer. Therefore, all the points $p_{j_2}$, $j_2\neq j_1$, belong to the interiors of pairwise distinct sectors of $p_{j_1}$, and a sufficiently small rotation does not affect this property, as claimed, see Figure~\ref{FYaok2}~(left). This completes the proof of Theorem~\ref{maxminomegaG} if $k=2m+1$ and $m$ is even, and hence in all cases.
\end{proof}

\section{Edge density} \label{sec:edges}

In any ordered Yao graph, the number of edges equals the sum of the outdegrees of the points, and the outdegree of a point is the number of non-empty sectors at the moment of its addition. Thus, we focus on this quantity in this section as the distances between points are irrelevant in this aspect. This also means that our results for the edge number of ordered Yao graphs also hold for, say, ordered theta graphs.

\subsection{Maximizing the number of edges}
\begin{proof}[Proof of Theorem~\ref{prop:minmaxEG}]
In the proof of Theorem~\ref{th:clique_minmax}, we established the following statement: for $k\neq 3$ and $n \geq 3$, every $n$-point set contains a point $p_n$ that does not contain all other points in a single sector. Moreover, the same statement also holds when $k = 3$ but $n \geq 4$. Therefore, when $k \neq 3$, we can recursively choose points $p_n,p_{n-1},\dots,p_4,p_3$ (and $p_2,p_1$ are arbitrary) such that $p_{n-i}$ does not contain all points $p_1,\dots,p_{n-i-1}$ in a single sector for all $i \leq n-3$. Thus, if we use the ordering $p_1,p_2,\dots,p_n$, the corresponding $k$-sector ordered Yao graph contains at least $0+1+2\cdot(n-2) = 2n-3$ edges, as desired. The case $k=3$ is similar, with the only difference that we recursively choose points until only $3$ points left (namely $p_1,p_2,p_3$). As a result, we construct an ordering such that the corresponding $3$-sector ordered Yao graph contains at least $0+1+1+2\cdot(n-3) = 2n-4$ edges, as desired. This proves the lower bound. As for the upper bound, observe that for $n$ points on a line (not parallel to the $\ell_i$, so the construction is in general position), the two endpoints both have one sector which contains all the other points, while the rest has two, thus no ordering can generate more than $2n-3$ edges, regardless of the value of $k$. In case $k=3$, Figure \ref{fig:2n-4edgesk3} provides a better construction: whatever ordering we take, $A$, $B$ and $C$ have outdegree at most $1$, while all other points have outdegree at most $2$. Combined with the fact that the first point always has outdegree $0$, this gives an upper bound of $2n-4$ on the number of edges.
\end{proof}
\begin{figure}[!htb]
		\centering
		\includegraphics{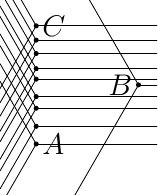}
		\caption{The points of the construction along with their incident rays.}
		\label{fig:2n-4edgesk3}
\end{figure}

\subsection{Minimizing the number of edges}

We prove Theorem~\ref{thm:maxminEG} in this subsection. The upper bound $e_k(n) \le (n-1)\cdot\left\lceil\frac{k}{2}\right\rceil$ is trivial. Indeed, the last $\left\lfloor\frac{k}{2}\right\rfloor$ of the sectors of each vertex belong to the lower half-plane, so they do not contain the preceding points in the top-to-bottom ordering. Hence, all vertices have outdegree at most $\left\lceil\frac{k}{2}\right\rceil$ in the corresponding $k$-sector ordered Yao graph. For $k=2$ this upper bound is trivially sharp. Thus, we assume $k\ge3$ in the following subsections.

\subsubsection{Upper bound: $e_k(n) \le n\cdot\left\lceil\frac{k}{2}\right\rceil -\left\lceil\sqrt{n}\right\rceil\cdot\left\lfloor\frac{k+1}{4}\right\rfloor$} \label{edges_upper}
	\begin{figure}[!htb]
		\centering
		\begin{minipage}{.69\textwidth}
		\begin{center}
		\centering
		\includegraphics[scale=0.8]{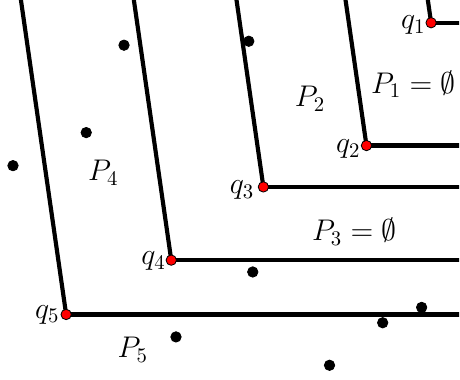}
		\end{center}
		\end{minipage}
		\begin{minipage}{.29\textwidth}
		\begin{center}
		\centering
		\includegraphics[scale=0.8]{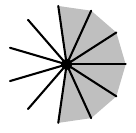}
		
		\includegraphics[scale=0.8]{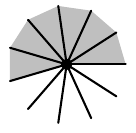}
		
		\includegraphics[scale=0.8]{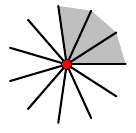}
		\end{center}
		\end{minipage}
		\caption{Left: the points in $Q$ are red; the rays $\ell_0$ and  
        $\ell_{\left\lceil\frac{k}{4}\right\rceil}$ at each $q_i$ are also drawn; and they divide $P\!\setminus\! Q$ into point sets $P_i$. Right: the sectors a newly added point ``sees''; the sectors in grey are possibly non-empty, and the sectors in white must be empty; the third case represents when the added point is from $Q$.}
		\label{fig:ErdosSzekeres}
	\end{figure}

Let $P$ be an arbitrarily given point set in general position, we need to find an ordering such that the corresponding ordered Yao graph has fewer than the claimed number of edges. For any $p \in P$, let $x(p)$ and $y(p)$ be the Cartesian coordinates of $p$. Consider the intersection of the $x$-axis with the line passing through $p$ along the direction of  $\ell_{\left\lceil\frac{k}{4}\right\rceil}$, and denote the $x$-coordinate of this intersection as $x'(p)$. Note that the general position hypothesis guarantees that $y(p) \neq y(q)$ and $x'(p) \neq x'(q)$ for $p\neq q$, although it is possible that $x(p) = x(q)$ unless $k$ is a multiple of $4$.

By the Erdős--Szekeres theorem, we can find a subset $Q=\left\lbrace q_1,...,q_m\right\rbrace$ for some $m\ge\left\lceil\sqrt{n}\right\rceil$ with $x'(q_i) > x'(q_{i+1})$ and satisfying either $y(q_i) > y(q_{i+1})$ or $y(q_i) < y(q_{i+1})$ for all $i$. Let us assume the former case $y(q_i) > y(q_{i+1})$ is true, the other case can be similarly handled as we remark at the end of the proof. Let $Q$ be maximal among all such subsets. We denote $P_i^+= \{ p \in P : x'(p) > x'(q_i),\ y(p)>y(q_i)\}$ and $P_{m+1}^+ = P\setminus Q$. Then, we define $P_i = P_{i+1}^+ \setminus P_i^+$. Notice that the maximality of $Q$ implies $P_1^+ = \emptyset$, so $P_1,\dots, P_m$ form a partition of $P \setminus Q$.

Now, we determine an ordering on the points of $P_i$ for a fixed $i$: First, take the points with a larger $y$ value than $y(q_i)$, ordered by their $x'$-values in decreasing order; then take the remaining points, ordered by their $y$ values in decreasing order. Let us denote this ordering of $P_i$ as 
$\text{Order}(P_i)$. Thus, we define the ordering of $P$ as follows:\begin{equation*}
    q_1,\text{Order}(P_1),q_2,\text{Order}(P_2),\dots, q_m,\text{Order}(P_m).
\end{equation*}

Next, we bound the number of edges in this ordering. Indeed, the points preceding $q_i$ have both a larger $x'$-value and a larger $y$-value coordinate than $q_i$, which means $q_i$ has outdegree at most $\left\lceil\frac{k}{4}\right\rceil$. For each point in $P_i$, it only sees preceding points from either right or above (see Figure~\ref{fig:ErdosSzekeres} right), so it has outdegree at most $\left\lceil\frac{k}{2}\right\rceil$. Therefore, we can conclude our upper bound using $m\ge\left\lceil\sqrt{n}\right\rceil$ and $\left\lfloor\frac{k+1}{4}\right\rfloor=\left\lceil\frac{k}{2}\right\rceil-\left\lceil\frac{k}{4}\right\rceil$.

Finally, if instead we had $y(q_i) < y(q_{i+1})$ for points in $Q$, we consider the consecutive sectors $s_{\left\lceil\frac{k}{4}\right\rceil}, s_{\left\lceil\frac{k}{4}\right\rceil+1},\dots, s_{2\cdot \left\lceil\frac{k}{4}\right\rceil-1}$ of $q_1,q_2,\dots,q_{m}$. They are nested just like the consecutive sectors $s_1,s_2,\dots,s_{\left\lceil\frac{k}{4}\right\rceil}$ of $q_1,q_2,\dots,q_m$ in the previous case (where $y(q_i) > y(q_{i+1})$). Hence, we can choose a similar ordering as before after a suitable rotation of the configuration.

\subsubsection{Lower bound: $e_k(n) \geq n \cdot \left\lceil\frac{k}{2}\right\rceil-O\left(\sqrt{n}\cdot k^2\right)$} \label{edges_lower}

First, we shall construct a point set $P$ of $n$ elements. Take a set of size $n$ that contains the $\left\lfloor\sqrt{n}\right\rfloor\times\left\lfloor\sqrt{n}\right\rfloor$ grid and is contained in a $\left\lceil\sqrt{n}\right\rceil\times\left\lceil\sqrt{n}\right\rceil$ grid. Apply a scaling transformation to make the gap of the grid slightly larger than $1$, but the diameter of the set remains at most $\left\lceil\sqrt{2}\cdot\left\lceil\sqrt{n}\right\rceil\right\rceil$. Make the point set in general position after small perturbations of the points, while ensuring the minimum distance among the points remains over $1$ and the diameter of the point set remains under $\left\lceil\sqrt{2}\cdot\left\lceil\sqrt{n}\right\rceil\right\rceil$. The resulting set is our $P$.

We shall show that the number of edges is always at least $n \cdot \left\lceil\frac{k}{2}\right\rceil-O\left(\sqrt{n}\cdot k^2\right)$ in the ordered Yao graph associated with $P$, regardless of the ordering. We fix an arbitrary ordering of $P$ and define a multiset $S$ containing vertices with outdegree less than $\left\lceil\frac{k}{2}\right\rceil$ in the ordered Yao graph constructed according to this ordering, where the multiplicity of each vertex is the difference between its outdegree and $\left\lceil\frac{k}{2}\right\rceil$. It suffices to prove $\lvert S\rvert\le O\left(\sqrt{n}\cdot k^2\right)$ where $|S|$ is counted with multiplicity.

Now, let $S_i$ be the set of points $p$ for which both sectors $s_i(p)$ and $s_{i+\lfloor k/2\rfloor}(p)$ are empty when $p$ is added in the ordered Yao graph. Here the addition of the indices is computed mod $k$. Observe that if a sector $s_j(p)$ is non-empty when $p$ is added, it would only disqualify $p$ to be in $S_j$ and $S_{j - \lfloor k/2\rfloor}$. This means the number of appearances of $p$ in $S_0,S_1,\dots,S_{k-1}$ is at least $k - 2x_p$ where $x_p$ is the outdegree of $p$. As a consequence, we can estimate \begin{equation*}
    \sum\limits_{i=0}^{k-1}{|S_i|} \geq \sum_{p \in \overline{S}} k - 2x_p = \sum_{p \in \overline{S}} 2\cdot \left(\frac{k}{2} - x_p\right) \geq \sum_{p \in \overline{S}} \frac{k}{2} - x_p + \frac{1}{2} \geq \sum_{p \in \overline{S}} \left\lceil\frac{k}{2}\right\rceil - x_p =|S|,
\end{equation*} where $\overline{S}$ is the set of elements in $S$ without multiplicity; recall that $x_p< \left\lceil\frac{k}{2}\right\rceil$ for each $p \in \overline{S}$, and thus $\frac{k}{2} - x_p \ge \frac{1}{2}$. Hence, it suffices for us to prove $|S_i|\le O\left(\sqrt{n}\cdot k\right)$ for all $i$.

\begin{figure}[!htb]
		\centering
		\includegraphics{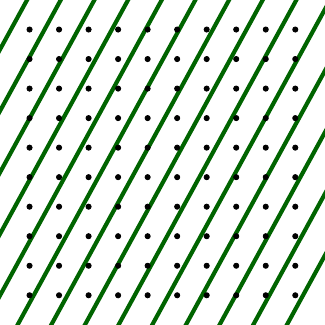}
		\caption{The points of $P$ along with the $\sigma_{i,j}$'s (their borders being denoted by green) for some fixed $i$.}
		\label{fig:Gridwithstrips}
\end{figure}
Next, we shall define strips $\sigma_{i,j}$ for $i = 0,\dots, k-1$ and $j = 1, \dots, t_i$ for some $t_i \leq O(\sqrt{n})$. We shall prove later that $S_i\cap\sigma_{i,j}=O(k)$ for any $i,j$, and that would conclude the proof. Consider the line $\ell$ bisecting the two rays $\ell_i$ and $\ell_{i + \lfloor k/2\rfloor+1}$ at an arbitrary point. Let $\sigma_i$ denote the minimal strip that contains $P$ and whose direction is perpendicular to $\ell$. Notice that the diameter of $P$ is at most $\left\lceil\sqrt{2}\cdot\left\lceil\sqrt{n}\right\rceil\right\rceil$, we can cover $\sigma_i$ by at most $O(\sqrt{n})$ many strips of width 1 along the same direction, and denote these strips by $\sigma_{i,j}$ for different $j$. See Figure~\ref{fig:Gridwithstrips}.

Now we shall prove $S_i\cap\sigma_{i,j}=O(k)$ with the help of the following geometric lemma.

\begin{lemma}\label{lem:strip}
Let $p$ be a point in the plane and $r_1$, $r_2$, $r_3$ and $r_4$ be distinct rays from $p$ counterclockwise in the aforementioned order with $\angle{r_ir_{i+1}}$ being called $\alpha_i$ and the closed cone defined by $r_i$ and $r_{i+1}$ by $C_i$ (the indices being counted mod 4). If $\alpha_1=\alpha_3$ and both $\alpha_2$ and $\alpha_4$ are smaller than $\pi$, then for any strip $\sigma$ of width $1$ perpendicular to the angle bisector $\ell$ of $C_2$ (and $C_4$) and containing $p$, all points of $\sigma\cap\left(C_2\cup C_4\right)$ have distance at most $\max\left\{\frac{1}{\cos{\left(\alpha_2/2\right)}},\frac{1}{\cos{\left(\alpha_4/2\right)}}\right\}$ from $p$.
\end{lemma}
\begin{proof}
Let the strip $\sigma$ intersects the rays $r_1,r_2,r_3$ and $r_4$ at points $a,b,c$ and $d$, respectively. Let $\ell$ intersects $\sigma$ in cones $C_2$ and $C_4$ at points $e$ and $f$, respectively. See Figure \ref{fig:strip}.
By construction, the right triangle $\Delta peb$ is congruent to $\Delta pec$. Similarly, $\Delta pfa$ is congruent to 
$\Delta pfd$.
To prove the lemma, it is sufficient to show that any point in $\Delta peb$ is at distance at most 
$\frac{1}{\cos{(\alpha_2/2)}}$ from $p$, since a similar argument holds for other points in 
$\sigma\cap\left(C_2\cup C_4\right)$. Now, for any point $q$ in $\Delta peb$, we have $|pq|\le |pb|=\frac{|ep|}{\cos{(\alpha_2}/2)}\leq\frac{1}{\cos{(\alpha_2}/2)}$, as desired.
\end{proof}

\begin{figure}[!htb]
		\centering
		\includegraphics[scale=0.8]{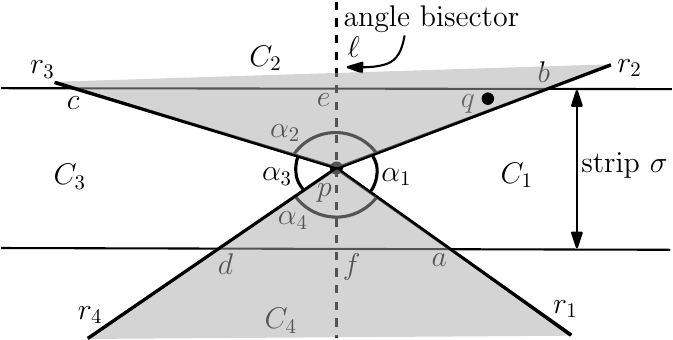}
		\caption{We shall apply the lemma with $r_1=\ell_i$, $r_2=\ell_{i+1}$, $r_3=\ell_{i+\lfloor k/2\rfloor}$, $r_4=\ell_{i+\lfloor k/2\rfloor+1}$, which makes $C_1$ equal to the sector $s_i$ and $C_3$ equal to the sector $s_{i+\lfloor k/2\rfloor}$.}
		\label{fig:strip}
\end{figure}

Now, let $p$ be the leftmost point and $q$ be the rightmost point among $S_i \cap \sigma_{i,j}$. (When $\sigma_{i,j}$ is vertical, we instead take $p$ as the lowest and $q$ as the highest.) Observe that either $q \not\in s_i(p)\cup s_{i+\lfloor k/2\rfloor}(p)$ or $p \not\in s_i(q)\cup s_{i+\lfloor k/2\rfloor}(q)$, otherwise the one that comes later in the ordering could not be in $S_i$. We argue that the distance between $p$ and $q$ is at most \begin{equation*}
    D := \frac{1}{\cos\left(\pi\cdot(\lceil k/2\rceil-1)/k\right)}.
\end{equation*} If $q \not\in s_i(p)\cup s_{i+\lfloor k/2\rfloor}(p)$, we apply Lemma~\ref{lem:strip} at $p$ with $r_1=\ell_i$, $r_2=\ell_{i+1}$, $r_3=\ell_{i+\lfloor k/2\rfloor}$, $r_4=\ell_{i+\lfloor k/2\rfloor+1}$, and $\sigma = \sigma_{i,j}$ to reach this conclusion; if $p \not\in s_i(q)\cup s_{i+\lfloor k/2\rfloor}(q)$, we apply Lemma~\ref{lem:strip} at $q$ instead, and deduce that the distance between $p$ and $q$ is at most $D$ either way. According to how $p$ and $q$ are chosen, this means that $S_i\cap \sigma_{i,j}$ is contained in a rectangle of size $1 \times D$. Also, since any two points in $P$ have distance larger than 1 by our construction, such a rectangle can contain at most $2 \lceil D\rceil$ points from $P$. Finally, one can check that $D=O(k)$ using elementary trigonometry, thus $S_i \cap \sigma_{i,j} \leq O(k)$, and thereby we conclude the proof.

\section{Concluding remarks} \label{sec:conc}

Our Theorems~\ref{thm:max_deg_Yao},~\ref{thm:maxminEG} and~\ref{maxminomegaG} leave room for improvements. For the first, it would be interesting to find out whether our simple bounds $\frac{1}{2}(n-1)\le d_3(n)$ and $d_5(n) \le n-1$ are tight or not.

In the setting of Theorem~\ref{maxminomegaG}, we suspect that the upper bound $w_k \le \left\lceil\frac{k}{2}\right\rceil+1$ may be tight since we confirmed it for $k=3,4$. However, we also showed that $w_k\left(\left\lceil\frac{k}{2}\right\rceil+1\right)=\left\lceil\frac{k}{2}\right\rceil$ for $k\ge4$. In other words, a point set for which a clique of size $\left\lceil\frac{k}{2}\right\rceil+1$ is unavoidable contains more than $\left\lceil\frac{k}{2}\right\rceil+1$ points.

In order to extend our results to point sets that are not necessarily in general position, we need to specify how ordered Yao graphs are constructed for such point sets. For this purpose, we need to handle two issues: a point $q$ being on the sector bounding ray of a point $p$; two points $q$ and $q'$ having the same distance to a point $p$. While there are no standard conventions, and perhaps the best practice depends on the potential applications of ordered Yao graphs, we still make some conventions from a mathematical point of view. We shall include the sector bounding ray $\ell_i(p)$ into the sector $s_i(p)$, hence a point $q$ in $\ell_i(p) \setminus p$ will be considered in the sector $s_i(p)$. For points $q$ and $q'$ both having the smallest distance to $p$ in $s_i(p)$, we shall break the tie arbitrarily, in particular, the ordered Yao graph construction now also involves tie-breaking choices.

Our results are of two kinds: a construction proof with a fixed point set $P$ such that the ordered Yao graph satisfies a certain bound regardless of the order on $P$; an argument proof for an arbitrary point set $P$ stating that there exists a suitable order on $P$ making the ordered Yao graph satisfying a certain bound. For the former kind of construction proofs, we gave point sets in general position, so their corresponding results naturally hold in the more general setting.

We explain why the statements corresponding to the latter argument proofs still hold for all point sets (pending suitable tie-breaking choices). Suppose $P$ is an arbitrary point set that is not necessarily in general position, we can slightly rotate $P$ counterclockwise to obtain $P'$ such that no point lies on the sector bounding ray of another point, and slightly perturb $P'$ to obtain $P''$ such that $P''$ contains no isosceles triangles. We can ensure that a point $q$ lies in the same sector of another point $p$ before and after the rotation and perturbation, because the rotation from $P$ to $P'$ is counterclockwise and both the rotation and the perturbation are very small. The perturbation from $P'$ to $P''$ essentially specifies a set of potential tie-breaking choices: if $q$ is closer to $p$ than $q'$ in $P''$, then $q$ is given priority than $q'$ in $P$ when a tie-breaking is needed. Now, if we have a statement saying there exists a suitable ordering on $P''$ such that the corresponding ordered Yao graph satisfies a certain bound, then using the same ordering and the tie-breaking choices inscribed in $P''$, we can construct the same ordered Yao graph on $P$, where the same bound is satisfied. This allows us to extend the results of this paper to point sets not necessarily in general position.

\paragraph{Acknowledgements.}
The authors would like to thank the organizers of the Focused Week on Geometric Spanners 
(Oct 23, 2023 -- Oct 29, 2023) at the Erd\H os Center, Budapest, where this joint work began. Research partially supported by ERC Advanced Grants `GeoScape' No. 882971 and `ERMiD' No. 101054936, by the Erd\H os Center, by the Ministry of Innovation and Technology NRDI Office within the framework of the Artificial Intelligence National Laboratory (RRF-2.3.1-21-2022-00004) and by grant no. 23-04949X of the Czech Science Foundation (GA\v{C}R). We are grateful to Géza Tóth for a helpful discussion on the result of Section~\ref{edges_lower}. We also thank the anonymous reviewers whose detailed reports helped to improve the paper.

\end{document}